\documentclass[11pt,a4paper]{amsart}

\RequirePackage[OT1]{fontenc}
\RequirePackage{amsthm,amsmath}
\RequirePackage[numbers]{natbib}
\RequirePackage[colorlinks,citecolor=blue,urlcolor=blue]{hyperref}

\usepackage{graphicx}
\newtheorem{theorem}{Theorem}
\newtheorem{lemma}[theorem]{Lemma}
\newtheorem*{lemma*}{Lemma}
\newtheorem{proposition}[theorem]{Proposition}
\newtheorem{corollary}[theorem]{Corollary}

\newtheorem{definition}[theorem]{Definition}
\newtheorem{remark}[theorem]{Remark}

\usepackage[utf8]{inputenc}
\usepackage[english]{babel}
\usepackage{latexsym}
\usepackage{amssymb}
\usepackage{amsmath}

\newcommand{\Z}{\mathbb{Z}}
\newcommand{\R}{\mathbb{R}}
\newcommand{\C}{\mathbb{C}}
\newcommand{\E}{\mathbb{E}}

\begin{document}

\title[Random unitary matrices and GMC]{The characteristic polynomial of a random unitary matrix and Gaussian multiplicative chaos - The $L^2$-phase.}
\author{Christian Webb}
\address{Department of mathematics and systems analysis, Aalto University, PO Box 11000, 00076 Aalto, Finland}
\email{christian.webb@aalto.fi}
\date{\today}

\begin{abstract} 
\ We study the characteristic polynomial of Haar distributed random unitary matrices. We show that after a suitable normalization, as one increases the size of the matrix, powers of the absolute value of the characteristic polynomial as well as powers of the exponential of its argument converge in law to a Gaussian multiplicative chaos measure for small enough real powers. This establishes a connection between random matrix theory and the theory of Gaussian multiplicative chaos.
\end{abstract}

\maketitle

\section{Introduction}

Studying the eigenvalues of the Circular Unitary Ensemble (CUE) - that is Haar distributed random unitary matrices - is a classical problem in random matrix theory \cite{dyson}. More recently it has gotten a lot of attention due to the conjectured relationship between the Riemann $\zeta$-function on the critical line $t\mapsto \frac{1}{2}+it$ and characteristic polynomials of large random matrices - namely it is believed that statistical properties of the $\zeta$ function evaluated at a random point on the critical line are related to the corresponding properties of the characteristic polynomial of a large random matrix, see e.g. \cite{ks}. The goal of this note is to describe the asymptotic behavior of the characteristic polynomial of a large Haar distributed unitary matrix when the characteristic polynomial is evaluated on the unit circle (where the eigenvalues lie).

\vspace{0.3cm}

There are of course existing results on the asymptotic behavior of the characteristic polynomial. For example, in \cite{ks}, it is shown that after normalizing by the variance, the logarithm of the characteristic polynomial at a single point is asymptotically Gaussian. This was refined in \cite{bhny}, where an exact decomposition for the law of the characteristic polynomial at a single point was given. On the other hand, in \cite{cnn}, it was shown that on the microscopic scale, the characteristic polynomial behaves like a random analytic function up to a normalization. The results that are closest to ours, and also the strongest motivation for this work, are those of Diaconis and Shahshahani \cite{ds} as well as Hughes, Keating, and O'Connell \cite{hko}, who proved among other things, that the real and imaginary parts of the logarithm of the characteristic polynomial restricted to the unit circle converge in law to a pair of Gaussian fields which can be represented as random generalized functions whose covariance kernel has a logarithmic singularity. 

\vspace{0.3cm}

In the 80s, Kahane constructed a theory for exponentiating such fields and understanding this exponential as a random multifractal measure \cite{kahane}. The theory is known as Gaussian Multiplicative Chaos (GMC). For a comprehensive review, see \cite{gmcrev}. Recently these measures have been of great interest due to their role in the mathematical study of two-dimensional quantum gravity (see e.g. \cite{kpz} for the physical motivation and \cite{dupshef,dkrv} for mathematical results). Thus conjecturally, these measures also play a role in the study of random planar maps (see e.g. \cite{dupshef,dkrv} for mathematical conjectures, and \cite{de} for a physical and historical point of view). Other applications of multiplicative chaos are construction of random planar curves through conformal welding \cite{ajks,shefweld}, Quantum Loewner Evolution \cite{qle}, studying properties of Gibbs measures of disordered systems \cite{cld}, energy dissipation in turbulence \cite{kol,obu}, and models for asset returns in mathematical finance \cite{bm}.

\vspace{0.3cm}

In \cite{fk}, Fyodorov and Keating essentially conjectured that as the size of the matrix tends to infinity, real powers of the absolute value of the characteristic polynomial of a CUE matrix converges to a GMC measure once suitably normalized. They then used this conjecture to make further conjectures about the absolute value of the characteristic polynomial and the $\zeta$-function on the critical line. Our main result will be that indeed, for small enough real powers, powers of the absolute value of the characteristic polynomial on the unit circle (as well as powers of the exponential of the argument of the characteristic polynomial) will converge in law to a GMC measure.

\vspace{0.3cm}

In addition to perhaps describing some properties of the $\zeta$-function, another motivation for this work is that this type of result can be seen as a new type of geometric limit theorem in the framework of random matrices. These types of results are likely to be rather universal in random matrix theory (see the discussion at the end of this paper), though to our knowledge it is the first of its kind.  As mentioned, limit theorems concern often a single point or the microscopic scale (or perhaps the mesoscopic scale). The global results of \cite{ds,hko} describe convergence to a rough object whose geometry is not easy to study. In fact, it seems that these measures are the correct way to study the geometry of the underlying Gaussian field. For example, these measures play a critical role in understanding the extrema of the field \cite{bdz,mad,drz}. Also the measures can be used to study the field's fractal properties (e.g. thick points of the field and a geometrical KPZ relation \cite{kahane,dupshef,kpzrv,gmcrev}). In \cite{bss} it was shown in the particular case of the Gaussian Free Field that this exponentiation does not lose any information about the field so all of the geometric properties of the field should be visible in the measure.

\vspace{0.3cm}

On the other hand, from the point of view of the theory of Gaussian Multiplicative Chaos itself, our result gives a very different type of construction of the measure than those common in the literature. Usually one uses Gaussian or even martingale approximations to the field which are essentially tailored to ensure convergence of the approximating measure. Here we have an approximation arising from a completely different model and one has no martingale property or Gaussianity until one passes to the limit. Thus the results here suggest that perhaps the measures are quite universal objects, or that this procedure of exponentiating a distribution is continuous in some sense, namely any "reasonable" approximation to the field should give a way to construct a GMC measure.

\vspace{0.3cm}

As the methods used in this paper are not that original (we use a natural approximation for the characteristic polynomial and a rather elementary approach to proving convergence coupled with powerful recent results on Toeplitz determinants in \cite{deift,claeyskrasovsky}), the main goal of this article is pointing out this connection between two important areas in modern probability theory and some of the interesting questions that this connection implies for both random matrix theory as well as the theory of Gaussian multiplicative chaos.

\vspace{0.3cm}

The outline of this paper is the following. We begin with recalling some facts and results about the CUE, describe our object of interest and state our main theorem and sketch the strategy of our proof. Next we discuss the relationship between the characteristic polynomial and Toeplitz determinants with Fisher-Hartwig singularities. We then review recent results from \cite{claeyskrasovsky,deift} on asymptotics of such Toeplitz determinants. Using these results we prove convergence to a GMC measure. Finally we discuss some open questions this result implies. For the convenience of the reader, we also have an appendix on Gaussian Multiplicative Chaos measures and Sobolev Spaces.

\section{The Circular Unitary Ensemble, the Main Result, and the Strategy of the Proof}

In this section, we will describe our basic model, object of interest, and main theorem as well as sketch the strategy for proving it.

\vspace{0.3cm}

As noted in the introduction, we are interested in $n\times n$-dimensional random matrices distributed according to the (unique) Haar probability measure on the unitary group $U(n)$. Let us denote such a matrix by $U_n$. By the Weyl integration formula applied to $U(n)$, the eigenvalues of $U_n$, which we denote by $(e^{i\theta_1},...,e^{i\theta_n})$ (with $\theta_i\in[0,2\pi)$), are distributed according to 

\begin{equation}\label{eq:cuelaw}
\frac{1}{n!}\prod_{k<j}|e^{i\theta_k}-e^{i\theta_j}|^{2}\prod_{k=1}^{n}\frac{d\theta_k}{2\pi}.
\end{equation}

We are interested in the characteristic polynomial of $U_n$, namely we evaluate it on the unit circle (where all of its zeros lie) and define

\begin{equation}
p_n(\theta)=\det(1-e^{-i\theta}U_n)=\prod_{k=1}^{n}(1-e^{i(\theta_k-\theta)}).
\end{equation}

To describe the asymptotic properties of $p_n(\theta)$, we study its absolute value and phase. It will turn out to be natural to consider suitable powers of these. More precisely, we introduce the following object, which will be the main object of interest in the rest of this article.

\begin{definition}\label{def:main}
For $\alpha,\beta\in \R$, $n\in \Z_+$ and $\theta\in[0,2\pi)$, let 

\begin{equation}
f_{n,\alpha,\beta}(\theta)=|p_n(\theta)|^{\alpha}e^{\beta \mathrm{Im}\log p_n(\theta)},
\end{equation}

\noindent where by $\mathrm{Im}\log p_n(\theta)$ we mean the branch of the logarithm where

\begin{equation}
\mathrm{Im}\log p_n(\theta)=\sum_{k=1}^{n}\mathrm{Im}\log (1-e^{i(\theta_k-\theta)})
\end{equation}

\noindent and 

\begin{equation}
\mathrm{Im}\log (1-e^{i(\theta_k-\theta)})\in\left(-\frac{\pi}{2},\frac{\pi}{2}\right].
\end{equation}

We also consider the random Radon measure on the unit circle defined by 

\begin{equation}
\mu_{n,\alpha,\beta}(d\theta)=\frac{f_{n,\alpha,\beta}(\theta)}{\E(f_{n,\alpha,\beta}(\theta))}d\theta.
\end{equation}

\end{definition}

We then recall a result from \cite{ds} concerning traces of powers of $U_n$.

\begin{theorem}[Diaconis and Shahshahani]\label{th:ds}
Let $(Z_i)_{i=1}^{\infty}$ be i.i.d. standard complex Gaussians, i.e. complex random variables whose real and imaginary parts are independent centered real Gaussians with variance $\frac{1}{2}$. Then for any fixed $k$, 

\begin{equation}
\left(\mathrm{Tr}U_n,\frac{1}{\sqrt{2}}\mathrm{Tr}U_n^{2}...,\frac{1}{\sqrt{k}}\mathrm{Tr}U_n^{k}\right)\stackrel{d}{\to}(Z_1,...,Z_k)
\end{equation}

\noindent as $n\to\infty$.

\end{theorem}

We also recall the following result from \cite{hko} where it was noted that Theorem \ref{th:ds} can be used to describe the asymptotic behavior of the logarithm of the characteristic polynomial. For the definition of the Sobolev space $\mathcal{H}^{-\epsilon}_0$, see the appendix.

\begin{theorem}[Hughes, Keating, and O'Connell]\label{th:hko}
For any $\epsilon>0$, the pair $(\log |p_n(\theta)|,\mathrm{Im}\log p_n(\theta))$ (where $\mathrm{Im}\log p_n(\theta)$ is interpreted as in Definition \ref{def:main}) converges in law in $\mathcal{H}_0^{-\epsilon}\times \mathcal{H}_0^{-\epsilon}$ to the pair of Gaussian fields $(X(\theta),\widehat{X}(\theta))$, where 

\begin{equation}
X(\theta)=\frac{1}{2}\sum_{k=1}^{\infty}\frac{1}{\sqrt{k}}(Z_k e^{ik\theta}+Z_k^{*}e^{-ik\theta}),
\end{equation}

\begin{equation}
\widehat{X}(\theta)=\frac{1}{2}\sum_{k=1}^{\infty}\frac{1}{\sqrt{k}}(i Z_k e^{ik\theta}-iZ_k^{*}e^{-ik\theta}),
\end{equation}

\noindent and $(Z_k)_{k=1}^{\infty}$ are i.i.d. standard complex Gaussians.
\end{theorem}

\begin{remark}
Note that as $iZ_k\stackrel{d}{=}Z_k$, we have $X\stackrel{d}{=}\widehat{X}$. Moreover, for real $\alpha,\beta$, the rotation invariance of the law of $Z_k$ implies that 

\begin{equation}
\alpha X+\beta \widehat{X}\stackrel{d}{=}\sqrt{\alpha^{2}+\beta^{2}}X.
\end{equation}

This does not imply that $X$ and $\widehat{X}$ are independent - they are not. For example, formally (one can make this precise if one wishes)

\begin{equation}
\E(X(\theta)\widehat{X}(\theta'))=\sum_{k=1}^\infty \frac{1}{k}\sin(k(\theta-\theta'))
\end{equation}

\noindent which is non-zero unless $|\theta-\theta'|=k\pi$ for some integer $k$. We also point out that (again formally though one can make this too precise with little effort)

\begin{equation}
\E(X(\theta)X(\theta'))=\frac{1}{2}\sum_{k=1}^{\infty}\frac{1}{k}\cos(k(\theta-\theta'))=-\frac{1}{2}\log|e^{i\theta}-e^{i\theta'}|.
\end{equation}
\end{remark}

Motivated by these remarks and Theorem \ref{th:hko}, we expect that in distribution, $f_{n,\alpha,\beta}$ should asymptotically behave like $e^{\sqrt{\alpha^{2}+\beta^{2}}X}$. The following theorem is our main result and makes this statement precise. For a proper definition of the measure $\mu_{\sqrt{\alpha^{2}+\beta^{2}}}(d\theta)$, see the appendix.

\begin{theorem}\label{th:main}
For $\alpha>-\frac{1}{2}$ and $\alpha^{2}+\beta^{2}<2$, the measure $\mu_{n,\alpha,\beta}(d\theta)$ converges in distribution in the space of Radon measures on the unit circle equipped with the topology of weak convergence to the (non-trivial) Gaussian multiplicative chaos measure $\mu_{\sqrt{\alpha^{2}+\beta^{2}}}(d\theta)$ which can be formally written as 

\begin{equation}
\mu_{\sqrt{\alpha^{2}+\beta^{2}}}(d\theta)=e^{\sqrt{\alpha^{2}+\beta^{2}}X(\theta)-\frac{\alpha^{2}+\beta^{2}}{2}\E(X(\theta)^{2})}d\theta.
\end{equation}

\end{theorem}

\bf Strategy of proof: \rm  Our starting point for the proof is the remark that the convergence of $\mu_{n,\alpha,\beta}$ in distribution to $\mu_{\sqrt{\alpha^{2}+\beta^{2}}}$ in the space of Radon measures on the unit circle with the topology of weak convergence is equivalent to 

\begin{equation}
\int_0^{2\pi}g(\theta)\mu_{n,\alpha,\beta}(d\theta)\stackrel{d}{\to}\int_0^{2\pi}g(\theta)\mu_{\sqrt{\alpha^{2}+\beta^{2}}}(d\theta),
\end{equation}

\noindent as $n\to \infty$ for each continuous non-negative function $g$ defined on the unit circle. For details on this, see e.g. Chapter 4 in \cite{kallenberg2}. We prove this by approximating $\mu_{n,\alpha,\beta}$ by truncating the Fourier series of the logarithm of $f_{n,\alpha,\beta}$. More precisely, we note that using the expansion of $\log (1-z)$, we have 

\begin{equation}
\log f_{n,\alpha,\beta}(\theta)\sim -\frac{1}{2}\sum_{j=1}^{\infty}\frac{1}{j}\left((\alpha-\beta i)\mathrm{Tr}U_n^{j}e^{-ij\theta}+(\alpha+\beta i)\mathrm{Tr} U_n^{-j}e^{ij\theta}\right).
\end{equation}

We then approximate $\log f_{n,\alpha,\beta}$ by truncating this series.

\begin{definition}
For $k,n\in \Z_+$, $\alpha,\beta\in \R$ and $\theta\in [0,2\pi)$, let 

\begin{equation}
f_{n,\alpha,\beta}^{(k)}(\theta)=e^{ -\frac{1}{2}\sum_{j=1}^{k}\frac{1}{j}\left((\alpha-\beta i)\mathrm{Tr}U_n^{j}e^{-ij\theta}+(\alpha+\beta i)\mathrm{Tr} U_n^{-j}e^{ij\theta}\right)}
\end{equation}

\noindent and 

\begin{equation}
\mu_{n,\alpha,\beta}^{(k)}(d\theta)=\frac{f_{n,\alpha,\beta}^{(k)}(\theta)}{\E(f_{n,\alpha,\beta}^{(k)}(\theta))}d\theta.
\end{equation}

\end{definition}

The idea now is to show that for any fixed continuous function $g$, as we let $n\to\infty$ and then $k\to \infty$, 

\begin{equation}
\int_0^{2\pi}g(\theta)\mu_{n,\alpha,\beta}^{(k)}(d\theta)-\int_0^{2\pi}g(\theta)\mu_{n,\alpha,\beta}(d\theta)
\end{equation}

\noindent tends to zero in distribution while in the same limit, $\int_0^{2\pi}g(\theta)\mu_{n,\alpha,\beta}^{(k)}(d\theta)$ tends to $\int_0^{2\pi}g(\theta)\mu_{\sqrt{\alpha^{2}+\beta^{2}}}(d\theta)$ in distribution. The first fact will be established through a variance estimate in the next section, where we make use of a Toeplitz determinant representation and results of \cite{deift,claeyskrasovsky}. The second fact follows from Theorem \ref{th:ds} and the definition of $\mu_{\sqrt{\alpha^{2}+\beta^{2}}}$.

\vspace{0.3cm}

Finally we note that it is reasonable to expect that the restriction in the values of the parameters $\alpha$ and $\beta$ is simply due to the method of our proof and convergence will hold for a larger set of values. For further discussion, see the last section of this paper.

\section{Variance estimates and asymptotics of Toeplitz determinants with Fisher-Hartwig singularities}

The goal of this section is to prove the following result:

\begin{proposition}\label{prop:var}

For $\alpha>-\frac{1}{2}$ and $\beta\in \R$ such that $\alpha^{2}+\beta^{2}<2$,

\begin{equation}
\lim_{k\to\infty}\limsup_{n\to\infty}\E\left(\left(\int_0^{2\pi}g(\theta)\mu_{n,\alpha,\beta}^{(k)}(d\theta)-\int_0^{2\pi}g(\theta)\mu_{n,\alpha,\beta}(d\theta)\right)^{2}\right)=0
\end{equation}

\noindent for any given continuous non-negative function $g$ defined on the unit circle.

\end{proposition}

Much of this section will be well known to experts of random matrix theory, but we give a detailed presentation for the benefit of readers less familiar with it.

\vspace{0.3cm}

Expanding the square in the expectation and using Fubini, we see that what is relevant is obtaining uniform asymptotics for $\E(f_{n,\alpha,\beta}^{(k)}(\theta)f_{n,\alpha,\beta}^{(k)}(\theta'))$, $\E(f_{n,\alpha,\beta}^{(k)}(\theta)f_{n,\alpha,\beta}(\theta'))$, and $\E(f_{n,\alpha,\beta}(\theta)f_{n,\alpha,\beta}(\theta'))$, as well as $\E(f_{n,\alpha,\beta}^{(k)}(\theta))$ and $\E(f_{n,\alpha,\beta}(\theta))$ for all values of $\theta$ and $\theta'$ (even as $\theta\to\theta'$). As we will see, all of these quantities can be represented as Toeplitz determinants and their asymptotic behavior follows from existing work. To see the Toeplitz determinant representation, let us first recall the Heine-Szeg\"o identity (see e.g. \cite{bd}).

\begin{theorem}[Heine-Szeg\"o identity]\label{th:hs}
Consider a function defined on the unit circle: $f(\phi)=\sum_{n\in \Z}f_n e^{in\phi}$ which is in $L^{1}(d\phi)$. Then if $(e^{i\theta_k})_{k=1}^{n}$ are the eigenvalues of a Haar distributed $n\times n$ unitary matrix, then 

\begin{equation}
\E\left(\prod_{k=1}^{n}f\left(\theta_k\right)\right)=D_{n-1}(f),
\end{equation}

\noindent where the Toeplitz determinant $D_{n-1}(f)$ is the determinant of the matrix 

\begin{equation}
\left(\begin{array}{cccc}
f_0 & f_1 & \cdots & f_{n-1}\\
f_{-1} & f_0 & \cdots & f_{n-2}\\
\vdots & \vdots & \ddots & \vdots \\
f_{-n+1} & f_{-n+2} & \cdots & f_0
\end{array}\right).
\end{equation}

\end{theorem}

\begin{remark}\label{rem:ti}
It follows for example from the translation invariance of the law of $(\theta_i)_{i=1}^{n}$, that for any fixed $\theta$, one also has

\begin{equation}
\E\left(\prod_{k=1}^{n}f\left(\theta+\theta_k\right)\right)=D_{n-1}(f)
\end{equation}

\noindent or if we denote by $f_\theta$, the translation of $f$ by $\theta$: $f_\theta(\phi)=f(\theta+\phi)$, then $D_{n-1}(f_\theta)=D_{n-1}(f)$.

\end{remark}

The following fact is a direct consequence of Theorem \ref{th:hs}:

\begin{lemma}\label{le:var}
\begin{align}
\notag \E&\left(\left(\int_0^{2\pi}g(\theta)\mu_{n,\alpha,\beta}^{(k)}(d\theta)-\int_0^{2\pi}g(\theta)\mu_{n,\alpha,\beta}(d\theta)\right)^{2}\right)\\
&= \frac{1}{\left(\E(f_{n,\alpha,\beta}^{(k)}(0))\right)^{2}}\int_0^{2\pi}\int_0^{2\pi}g(\theta)g(\theta')D_{n-1}(\sigma_{1,\theta,\theta'})d\theta d\theta'\\
&\qquad -2\frac{1}{\E(f_{n,\alpha,\beta}^{(k)}(0))\E(f_{n,\alpha,\beta}(0))} \int_0^{2\pi}\int_0^{2\pi}g(\theta)g(\theta')D_{n-1}(\sigma_{2,\theta,\theta'})d\theta d\theta'\notag\\
&\qquad +\frac{1}{\left(\E(f_{n,\alpha,\beta}(0))\right)^{2}}\int_0^{2\pi}\int_0^{2\pi}g(\theta)g(\theta')D_{n-1}(\sigma_{3,\theta,\theta'})d\theta d\theta',\notag
\end{align}

\noindent where 

\begin{equation}
\sigma_{1,\theta,\theta'}(\phi)=e^{-\frac{1}{2}\sum_{j=1}^{k}\frac{1}{j}\left((\alpha-\beta i)(e^{-ij\theta}+e^{-ij\theta'})e^{ij\phi}+(\alpha+\beta i)(e^{ij\theta}+e^{ij\theta'})e^{-ij\phi}\right)},
\end{equation}

\begin{align}
\notag \sigma_{2,\theta,\theta'}(\phi)&=e^{-\frac{1}{2}\sum_{j=1}^{k}\frac{1}{j}\left((\alpha-\beta i)e^{-ij\theta}e^{ij\phi}+(\alpha+\beta i)e^{ij\theta}e^{-ij\phi}\right)}\\
&\qquad\times|e^{i\theta'}-e^{i\phi}|^{\alpha}e^{\beta \mathrm{Im}\log (1-e^{i(\phi-\theta')})}
\end{align}

\noindent and 

\begin{equation}
\sigma_{3,\theta,\theta'}(\phi)=|e^{i\theta}-e^{i\phi}|^{\alpha}e^{\beta \mathrm{Im}\log (1-e^{i(\phi-\theta)})}|e^{i\theta'}-e^{i\phi}|^{\alpha}e^{\beta \mathrm{Im}\log (1-e^{i(\phi-\theta')})},
\end{equation}

\noindent where the branch of the logarithm is such that $\mathrm{Im}\log (1-e^{i(\phi-\theta')})\in(-\frac{\pi}{2},\frac{\pi}{2}]$ and similarly for $\theta$.

\end{lemma}

\begin{proof}
This follows from applying Theorem \ref{th:hs} to the remark that for any $\theta,\theta'\in[0,2\pi)$

\begin{equation}
f_{n,\alpha,\beta}^{(k)}(\theta)f_{n,\alpha,\beta}^{(k)}(\theta')=\prod_{p=1}^{n}\sigma_{1,\theta,\theta'}(\theta_p)
\end{equation}

\noindent and similar arguments for $f_{n,\alpha,\beta}^{(k)}(\theta)f_{n,\alpha,\beta}(\theta')$ and $f_{n,\alpha,\beta}(\theta)f_{n,\alpha,\beta}(\theta')$. Remark \ref{rem:ti} implies that the denominators $\E(f_{n,\alpha,\beta}^{(k)}(\theta))$ are independent of $\theta$ and can be taken outside of the integrals.
\end{proof}

\begin{remark}
Due to our choice of the branch of the logarithm, we have

\begin{equation}
\mathrm{Im}\log (1-e^{i(\phi-\theta)})=\begin{cases}
-\frac{\pi}{2}+\frac{\phi-\theta}{2}, & 0\leq \theta\leq \phi<2\pi\\
\frac{\pi}{2}+\frac{\phi-\theta}{2}, & 0\leq \phi<\theta<2\pi
\end{cases}
\end{equation}

\noindent implying that we can write

\begin{align}
\notag\sigma_{2,\theta,\theta'}(\phi)&=e^{-\frac{1}{2}\sum_{j=1}^{k}\frac{1}{j}\left((\alpha-\beta i)e^{-ij\theta}e^{ij\phi}+(\alpha+\beta i)e^{ij\theta}e^{-ij\phi}\right)}\\
&\qquad\times|e^{i\theta'}-e^{i\phi}|^{\alpha}e^{\beta\frac{\phi-\theta'}{2}}g_{e^{i\theta'},-i\frac{\beta}{2}}(e^{i\phi})
\end{align}

\noindent and 

\begin{equation}
\sigma_{3,\theta,\theta'}(\phi)=|e^{i\theta}-e^{i\phi}|^{\alpha}e^{\beta\frac{\phi-\theta}{2}}g_{e^{i\theta},-i\frac{\beta}{2}}(e^{i\phi})|e^{i\theta'}-e^{i\phi}|^{\alpha}e^{\beta\frac{\phi-\theta'}{2}}g_{e^{i\theta'},-i\frac{\beta}{2}}(e^{i\phi}),
\end{equation}

\noindent where 

\begin{equation}\label{def:g}
g_{e^{i\theta},\beta}(e^{i\phi})=\begin{cases}
e^{i\pi \beta}, & 0\leq \phi<\theta\\
e^{-i\pi \beta}, &\theta\leq \phi< 2\pi
\end{cases}.
\end{equation}
\end{remark}

\noindent In the definition of $g_{e^{i\theta},\beta}$ we have followed the notation of \cite{deift} to avoid confusion when referring to their results. 

\vspace{0.3cm}

The asymptotics of such Toeplitz determinants have been studied extensively. The pointwise asymptotics of $D_{n-1}(\sigma_{1,\theta,\theta'})$ go back to Szeg\"o \cite{sze}. $D_{n-1}(\sigma_{2,\theta,\theta'})$ and $D_{n-1}(\sigma_{3,\theta,\theta'})$ are special cases of Toeplitz determinants with Fisher-Hartwig singularities. Conjectures about their asymptotic behavior go back to Fisher and Hartwig \cite{fh} as well as Lenard \cite{lenard}. The first rigorous results are due to Widom \cite{widom} though there is still a lot of research activity related to the problem (see e.g. \cite{ehrhardt1,deift,deift2,claeyskrasovsky}). Let us now discuss the asymptotics of the different terms.

\subsection{Asymptotics of $D_{n-1}(\sigma_1)$}

As noted, the pointwise asymptotics of such a determinant go back to Szeg\"o:

\begin{theorem}[Strong Szeg\"o theorem]\label{th:sz}
Let $L$ be a real valued function on the unit circle such that $L\in L^{1}$, $e^{L}\in L^{1}$ and let $\widehat{L}_k$ denote the Fourier coefficients of $L$: $\widehat{L}_n=\int_0^{2\pi}e^{-in\phi}L(\phi)\frac{d\phi}{2\pi}$. Then 

\begin{equation}
\log D_{n-1}(e^{L})= n \widehat{L}_0+\sum_{k=1}^{\infty}k|\widehat{L}_k|^{2}+\mathit{o}(1).
\end{equation} 
\end{theorem}

As it is, this is not quite enough for us. In our case, $L$ depends on the variables $\theta$ and $\theta'$ which we wish to integrate over so we need a uniform version for this. Actually as $\sigma_{1,\theta,\theta'}$ is real and $\int_0^{2\pi}\log \sigma_{1,\theta,\theta'}(\phi)d\phi=0$, $D_{n-1}(\sigma_1)$ is increasing in $n$: 

\begin{theorem}
Let $L$ be as in the previous theorem with the extra condition that $\widehat{L}_0=0$. Then for any $n\in \Z_+$, 
\begin{equation}
D_{n-1}(e^{L})\leq D_{n}(e^{L}).
\end{equation}
\end{theorem}

\begin{proof}
This is proven for example in \cite{simon}. More precisely, in Section 2 (Theorems 2.1-2.4) of \cite{simon} it is proven that 

\begin{equation}
\lim_{n\to\infty}\left(D_{n-1}(e^{L})\right)^{\frac{1}{n}}=\lim_{n\to\infty}\frac{D_{n}(e^{L})}{D_{n-1}(e^{L})}
\end{equation}

\noindent and that if this limit (which following \cite{simon} we denote by $F$) is positive, then for some increasing sequence $(G_n)$,

\begin{equation}
D_{n-1}(e^{L})=G_{n-1}F^{n} 
\end{equation}

As noted in Theorem 5.1 of \cite{simon}, it then follows from these results and Szeg\"o's theorem (the "weaker" one i.e. that $\frac{1}{n}\log D_{n-1}(e^{L})=\widehat{L}_0+\mathit{o}(1)$ - Theorem 4.1 of \cite{simon}) that $F=e^{\widehat{L}_0}=1$ and $D_{n-1}(e^{L})=G_{n-1}$ is increasing.

\end{proof}

Thus these two theorems and the dominated convergence theorem imply the following asymptotic behavior:

\begin{corollary}\label{cor:sigma1}
For any $\alpha,\beta\in \R$, and continuous $g$

\begin{align}
\notag &\lim_{n\to\infty}\int_0^{2\pi}\int_0^{2\pi}g(\theta)g(\theta')D_{n-1}(\sigma_{1,\theta,\theta'})d\theta d\theta'\\
&\qquad=\int_{0}^{2\pi}\int_0^{2\pi}g(\theta)g(\theta')e^{\frac{1}{4}\sum_{j=1}^{k}\frac{1}{j}(\alpha^{2}+\beta^{2})|e^{ij\theta}+e^{ij\theta'}|^{2}}d\theta d\theta'\\
\notag &\qquad =e^{\frac{1}{2}(\alpha^{2}+\beta^{2})\sum_{j=1}^{k}\frac{1}{j}}\int_{0}^{2\pi}\int_0^{2\pi}g(\theta)g(\theta')e^{\frac{(\alpha^{2}+\beta^{2})}{2}\sum_{j=1}^{k}\frac{1}{j}\cos (j(\theta-\theta'))}d\theta d\theta'.
\end{align}
\end{corollary}

\subsection{Asymptotics of $D_{n-1}(\sigma_2)$} The asymptotic behavior of determinants of the form of $D_{n-1}(\sigma_2)$ was already analyzed in \cite{widom} and generalized in \cite{basor} and \cite{ehrhardt}. Nevertheless, we shall formulate the results in terms of those of \cite{deift} as similar notations are used in \cite{claeyskrasovsky} which we shall rely on for the asymptotics of $D_{n-1}(\sigma_3)$.

\vspace{0.3cm}

As noted, $\sigma_2$ and $\sigma_3$ have Fisher-Hartwig singularities, namely they are both of the form

\begin{equation}\label{eq:fh}
f(z)=e^{V(z)}z^{\sum_{j=0}^{m}\beta_j}\prod_{j=0}^{m}|z-z_j|^{2\alpha_j}g_{z_j,\beta_j}(z)z_j^{-\beta_j},
\end{equation}

\noindent where $z=e^{i\phi}$ and $z_j$ are some fixed distinct points on the unit circle, in our notation they correspond to $e^{i\theta}$ and $e^{i\theta'}$, and $g_{z_j,\beta_j}$ was defined in \eqref{def:g}. 

\vspace{0.3cm}

For $\sigma_2$ the exact correspondence is the following: $m=0$, $\alpha_0=\frac{\alpha}{2}$, $z_0=e^{i\theta'}$, $\beta_0=-i\frac{\beta}{2}$, and 

\begin{equation}
V(z)=-\frac{1}{2}\sum_{j=1}^{k}\frac{1}{j}(\alpha-\beta i)e^{-ij\theta} z^{j}-\frac{1}{2}\sum_{j=1}^{k}\frac{1}{j}(\alpha+\beta i)e^{ij\theta} \overline{z}^{j}.
\end{equation}

In \cite{deift} a normalization is chosen where $z_0=1$, but making use of Remark \ref{rem:ti}, we can recover this by shifting $\theta\mapsto \theta-\theta'$, $\phi-\theta'\to \phi$, and $\theta'\to 0$.

\vspace{0.3cm}

The main result of \cite{deift} (proven in \cite{ehrhardt1} in the case where $V\in C^{\infty}$ - that is infinitely differentiable) is

\begin{theorem}[Ehrhardt; Deift, Its, and Krasovsky]\label{th:fh}
Let $f$ be of the form \eqref{eq:fh} and let $|||\beta|||:=\max_{j,k}|\mathrm{Re}\beta_j-\mathrm{Re}\beta_k|<1$, $\mathrm{Re}\alpha_j>-\frac{1}{2}$, $\alpha_j\pm \beta_j\neq -1,-2,...$ for $j=0,...,m$ and let $V(z)=\sum_{k\in \Z} V_k z^{k}$ satisfy 

\begin{equation}
\sum_{k\in \Z}|k|^{s}|V_k|<\infty
\end{equation}

\noindent for 

\begin{equation}
s>\frac{1+\sum_{j=0}^{m}((\mathrm{Im}\alpha_j)^{2}+(\mathrm{Re}\beta_j)^{2})}{1-|||\beta|||}.
\end{equation}

Then as $n\to\infty$, for $z_i\neq z_j$ for all $i\neq j$,

\begin{align}
\notag D_n(f)&=e^{nV_0+\sum_{k=1}^{\infty}kV_kV_{-k}}\prod_{j=0}^{m}e^{(\beta_j-\alpha_j)\sum_{k=1}^{\infty}V_k z_j^{k}}e^{-(\alpha_j+\beta_j)\sum_{k=1}^{\infty}V_{-k}z_j^{-k}}\\
&\qquad\times n^{\sum_{j=0}^{m}(\alpha_j^{2}-\beta_j^{2})}\prod_{0\leq j<k\leq m}|z_j-z_k|^{2(\beta_j\beta_k-\alpha_j\alpha_k)}\left(\frac{z_k}{z_j e^{i\pi}}\right)^{\alpha_j\beta_k-\alpha_k\beta_j}\\
&\qquad\notag\times \prod_{j=0}^{m}\frac{G(1+\alpha_j+\beta_j)G(1+\alpha_j-\beta_j)}{G(1+2\alpha_j)}(1+\mathit{o}(1)),
\end{align}

\noindent where $G$ is the Barnes $G$-function and the product $\prod_{0\leq j<k\leq m}$ is set to $1$ if $m=0$.

\end{theorem} 

\begin{remark}\label{rem:unif}
One can show that the error term is uniform in compact subsets of $\lbrace z_i\neq z_j\rbrace$: see e.g. \cite[Remark 1.4]{deift}. For $\sigma_2$ and $\sigma_3$ this can be seen also from the proofs of \cite{ehrhardt,widom}. More precisely, looking at the proof in \cite{ehrhardt} for the asymptotics corresponding to $\sigma_2$, one sees from the end of the proof (\cite[p. 254]{ehrhardt}) that the crucial estimate for uniformity is a uniform bound on the trace norm of the operator $A$ (defined on \cite[p. 249]{ehrhardt}). This then translates (through \cite[propositions 4.2 and 4.5]{ehrhardt}) into regularity conditions on the potential $V$ which in our case is uniformly bounded and all of its derivatives are uniformly bounded and one is able to prove uniform bounds on the trace norm of $A$. For $\sigma_3$, one can trace through the proof of \cite{widom} and uniform estimates boil down to the partial sums of $\sum_{l}(\frac{z_i}{z_j})^{l}$ are uniformly bounded in say $|z_i-z_j|\geq \epsilon$ - see \cite[p. 345]{widom} for the relevance of this estimate.
\end{remark}

\vspace{0.3cm}

Plugging in the values corresponding to $\sigma_2$ and shifting $\theta'\to 0$, $\theta\to\theta-\theta'$ (i.e. setting $\beta_0=-i\frac{\beta}{2}$, $\alpha_0=\frac{\alpha}{2}$, $V_j=-\frac{1}{2}(\alpha-\beta i)e^{-ij(\theta-\theta')}$, $V_{-j}=\overline{V_j}$), we see that  

\begin{align}
\notag D_{n-1}(\sigma_{2,\theta,\theta'})&=e^{\frac{1}{4}(\alpha^{2}+\beta^{2})\sum_{j=1}^{k}\frac{1}{j}}e^{\frac{1}{2}(\alpha^{2}+\beta^{2})\sum_{j=1}^{k}\frac{1}{j}\cos j(\theta-\theta')}n^{\frac{\alpha^{2}+\beta^{2}}{4}}\\
&\qquad\times \frac{G\left(1+\frac{\alpha}{2}-i\frac{\beta}{2}\right)G\left(1+\frac{\alpha}{2}+i\frac{\beta}{2}\right)}{G(1+\alpha)}(1+\mathit{o}(1)).
\end{align}

As, there is only one Fisher-Hartwig singularity in $\sigma_2$, we see by Remark \ref{rem:unif} that the error is uniform in $\theta,\theta'$. Thus we have

\begin{corollary}\label{cor:sigma2}
For any continuous function $g$ defined on the unit circle, $\alpha>-1$, and $\beta\in \R$, as $n\to\infty$

\begin{align}
\notag\int_0^{2\pi}&\int_0^{2\pi}g(\theta)g(\theta')D_{n-1}(\sigma_{2,\theta,\theta'})d\theta d\theta'&\\
&=n^{\frac{\alpha^{2}+\beta^{2}}{4}}\frac{G\left(1+\frac{\alpha}{2}-i\frac{\beta}{2}\right)G\left(1+\frac{\alpha}{2}+i\frac{\beta}{2}\right)}{G(1+\alpha)}e^{\frac{1}{4}(\alpha^{2}+\beta^{2})\sum_{j=1}^{k}\frac{1}{j}}\\
\notag &\qquad \times\left(\int_0^{2\pi}\int_0^{2\pi}g(\theta)g(\theta')e^{\frac{1}{2}(\alpha^{2}+\beta^{2})\sum_{j=1}^{k}\frac{1}{j}\cos j(\theta-\theta')}d\theta d\theta'+\mathit{o}(1)\right).
\end{align}

\end{corollary}

\subsection{Asymptotics of $D_{n-1}(\sigma_3)$} We again have a Toeplitz determinant with Fisher-Hartwig singularities. Compared to \eqref{eq:fh}, the relationship is $V=0$, $m=1$, $z_0=e^{i\theta}$,  $z_1=e^{i\theta'}$, $\alpha_0=\alpha_1=\frac{\alpha}{2}$, and $\beta_0=\beta_1=-i\frac{\beta}{2}$, or shifting to the normalization of Theorem \ref{th:fh}, $z_0=1$ and $z_1=e^{i(\theta'-\theta)}$. Theorem \ref{th:fh} and Remark \ref{rem:unif} then imply that for any $\epsilon>0$,

\begin{equation}
\lim_{n\to\infty}\frac{D_{n-1}(\sigma_{3,\theta,\theta'})}{n^{\frac{1}{2}(\alpha^{2}+\beta^{2})}}=\left|e^{i\theta}-e^{i\theta'}\right|^{-\frac{\alpha^{2}+\beta^{2}}{2}}\frac{G\left(1+\frac{\alpha}{2}-\frac{\beta}{2}i\right)^{2}G\left(1+\frac{\alpha}{2}+\frac{\beta}{2}i\right)^{2}}{G(1+\alpha)^{2}}
\end{equation}

\noindent uniformly in $|\theta-\theta'|\geq \epsilon$.

\vspace{0.3cm}

Compared to $D_{n-1}(\sigma_2)$ we have here the important difference that we must also consider the situation $\theta\to \theta'$ and we can't simply rely on Theorem \ref{th:fh}. 

\vspace{0.3cm}

Luckily the situation where $\theta\to \theta'$ has recently been analyzed in \cite{claeyskrasovsky}. In fact, the following is essentially their proof of Theorem 1.8 in \cite{claeyskrasovsky}, but as on a superficial level, our setting looks slightly more general, we write down the details. Specifying their Theorem 1.5 into our setting ($\alpha_1=\alpha_2=\frac{\alpha}{2}$, $\beta_1=\beta_2=-i\frac{\beta}{2}$) and ignoring the finer asymptotics that aren't needed for our result, we have the following

\begin{theorem}[Claeys and Krasovsky]\label{th:ck1}
There exists a $t_0>0$ such that for $\alpha>-\frac{1}{2}$ and $0<|\theta-\theta'|<2t_0$, 

\begin{align}
\log D_{n-1}(\sigma_{3,\theta,\theta'})&=\log D_{n-1}(\sigma_{3,0,0})+\int_0^{-in|\theta-\theta'|}\frac{1}{s}\left(\sigma(s)-\frac{1}{2}(\alpha^{2}+\beta^{2})\right)ds\notag\\
&\qquad -\frac{1}{2}(\alpha^{2}+\beta^{2})\log \frac{2\sin \frac{|\theta-\theta'|}{2}}{|\theta-\theta'|}+\mathit{o}(1),
\end{align}

\noindent where the integral is along the negative imaginary axis, $\mathit{o}(1)$ is uniform in $0<|\theta-\theta'|<2t_0$, and

\begin{equation}
\log D_{n-1}(\sigma_{3,0,0})=(\alpha^{2}+\beta^{2})\log n+\log \frac{G(1+\alpha-i\beta)G(1+\alpha+i\beta)}{G(1+2\alpha)}+\mathit{o}(1).
\end{equation}

Moreover $\sigma$ is a continuous function (depending only on $\alpha$ and $\beta$ - not $\theta,\theta',$ or $n$) whose asymptotic behavior is the following: there is some $\delta>0$ such that

\begin{equation}
\sigma(s)=\frac{1}{2}(\alpha^{2}+\beta^{2})+\mathcal{O}(|s|^{\delta}),
\end{equation}

\noindent as $s\to 0$ along the negative imaginary axis, and

\begin{equation}
\sigma(s)=\mathcal{O}(|s|^{-\delta})
\end{equation}

\noindent as $s\to\infty$ along the negative imaginary axis.
\end{theorem}

We shall also make use of their Theorem 1.11 which in our situation simplifies to the following.

\begin{theorem}[Claeys and Krasovsky]\label{th:ck2}
Let $\alpha>-1$. Then there exists a sufficiently small $t_0$ such that for $\frac{\log n}{n}\leq |\theta-\theta'|<2t_0$

\begin{align}
\notag \log D_{n-1}(\sigma_{3,\theta,\theta'})&=\frac{1}{2}(\alpha^{2}+\beta^{2})\log n-\frac{1}{2}(\alpha^{2}+\beta^{2})\log \left(2\sin\frac{|\theta-\theta'|}{2}\right)\\
&\qquad +\log \frac{G(1+\frac{\alpha}{2}-i\frac{\beta}{2})^{2}G(1+\frac{\alpha}{2}+i\frac{\beta}{2})^{2}}{G(1+\alpha)^{2}}+\mathit{o}(1)
\end{align}

\noindent and the error term is uniform in $\frac{\log n}{n}\leq |\theta-\theta'|<t_0$.

\end{theorem}

Combining these results we have the following asymptotics (essentially Theorem 1.15 of \cite{claeyskrasovsky}):

\begin{corollary}\label{cor:sigma3}
For any continuous function $g$ defined on the unit circle, $\alpha>-\frac{1}{2}$, and $\alpha^{2}+\beta^{2}<2$

\begin{align}
\notag\lim_{n\to\infty}&n^{-\frac{\alpha^{2}+\beta^{2}}{2}}\int_0^{2\pi}\int_0^{2\pi}g(\theta)g(\theta')D_{n-1}(\sigma_{3,\theta,\theta'})d\theta d\theta'\\
&=\frac{G(1+\frac{\alpha}{2}-i\frac{\beta}{2})^{2}G(1+\frac{\alpha}{2}+i\frac{\beta}{2})^{2}}{G(1+\alpha)^{2}}\\
\notag &\qquad \times\int_0^{2\pi}\int_0^{2\pi}g(\theta)g(\theta')|e^{i\theta}-e^{i\theta'}|^{-\frac{\alpha^{2}+\beta^{2}}{2}}d\theta d\theta'.
\end{align}

\end{corollary}

\begin{proof}
 Let us split the $\theta,\theta'$ integrals into four parts: $I_1$, being the integral over $0<|\theta'-\theta|\leq \frac{1}{n}$, $I_2$ corresponding to $\frac{1}{n}<|\theta'-\theta|<\frac{\log n}{n}$, $I_3$ corresponding to $\frac{\log n}{n}\leq |\theta'-\theta|<2t_0$, and $I_4$ corresponding to $2t_0\leq |\theta'-\theta|$.

\vspace{0.3cm}

By Theorem \ref{th:fh} and Remark \ref{rem:unif}, we have 

\begin{align}
\notag\lim_{n\to\infty}n^{-\frac{\alpha^{2}+\beta^{2}}{2}}I_4&=\frac{G(1+\frac{\alpha}{2}-i\frac{\beta}{2})^{2}G(1+\frac{\alpha}{2}+i\frac{\beta}{2})^{2}}{G(1+\alpha)^{2}}\\
&\qquad \times\int_{|\theta-\theta'|\geq 2t_0}g(\theta)g(\theta')|e^{i\theta}-e^{i\theta'}|^{-\frac{\alpha^{2}+\beta^{2}}{2}}d\theta d\theta'.
\end{align} 

For $I_1$, we note that Theorem \ref{th:ck1} implies that

\begin{align}
&\notag I_1=n^{\alpha^{2}+\beta^{2}}\frac{G(1+\alpha-i\beta)G(1+\alpha+i\beta)}{G(1+2\alpha)}\\
&\qquad \times\int_{|\theta-\theta'|\leq \frac{1}{n}}g(\theta)g(\theta')e^{\int_0^{-in|\theta-\theta'|}\frac{1}{s}\left(\sigma(s)-\frac{1}{2}(\alpha^{2}+\beta^{2})\right)ds}\\
\notag &\qquad \times e^{-\frac{1}{2}(\alpha^{2}+\beta^{2})\log \frac{2\sin \frac{|\theta-\theta'|}{2}}{|\theta-\theta'|}+\mathit{o}(1)}d\theta d\theta'.
\end{align} 
 
Moreover, the asymptotics of $\sigma$ near zero on the negative imaginary axis, imply that the integrand in the exponential converges and the integrand in the $\theta,\theta'$-integral is bounded, so we see that $I_1=\mathcal{O}(n^{\alpha^{2}+\beta^{2}-1})$ and as $\frac{\alpha^{2}+\beta^{2}}{2}<1$, this implies that $n^{-\frac{1}{2}(\alpha^{2}+\beta^{2})}I_1\to 0$ as $n\to \infty$. 
 
\vspace{0.3cm}

For $I_2$, using Theorem \ref{th:ck1} we write for $\frac{1}{n}< |\theta-\theta'|<\frac{\log n}{n}$

\begin{align}
\notag\log &D_{n-1}(\sigma_{3,\theta,\theta'})\\
\notag &=(\alpha^{2}+\beta^{2})\log n+\log \frac{G(1+\alpha-i\beta)G(1+\alpha+i\beta)}{G(1+2\alpha)}\\
&\qquad +\int_0^{-i}\frac{\sigma(s)-\frac{1}{2}(\alpha^{2}+\beta^{2})}{s}ds-\frac{1}{2}(\alpha^{2}+\beta^{2})\log n\\
\notag &\qquad+\int_{-i}^{-in|\theta-\theta'|}\frac{\sigma(s)}{s}ds-\frac{\alpha^{2}+\beta^{2}}{2}\log \left(2 \sin\frac{|\theta-\theta'|}{2}\right)+\mathit{o}(1)\notag
\end{align} 

\noindent and we have

\begin{align}
\notag I_2&=n^{\frac{\alpha^{2}+\beta^{2}}{2}}\frac{G(1+\alpha-i\beta)G(1+\alpha+i\beta)}{G(1+2\alpha)}e^{\int_0^{-i}\frac{\sigma(s)-\frac{1}{2}(\alpha^{2}+\beta^{2})}{s}ds}\\
&\qquad \times \int_{\frac{1}{n}\leq |\theta-\theta'|\leq\frac{\log n}{n}}g(\theta)g(\theta')\left(2\sin\frac{|\theta-\theta'|}{2}\right)^{-\frac{\alpha^{2}+\beta^{2}}{2}}\\
\notag &\qquad \times e^{\int_{-i}^{-in|\theta-\theta'|}\frac{\sigma(s)}{s}ds+\mathit{o}(1)}d\theta d\theta'.
\end{align} 
 
The asymptotics of $\sigma(s)$ as $s\to\infty$ along the negative imaginary axis imply that the integrand can be bounded by a constant times $|\theta-\theta'|^{-\frac{\alpha^{2}+\beta^{2}}{2}}$ which is an integrable singularity as $\frac{\alpha^{2}+\beta^{2}}{2}<1$. We conclude that as $n\to\infty$, $n^{-\frac{\alpha^{2}+\beta^{2}}{2}}I_2\to 0$.

\vspace{0.3cm}

For $I_3$, we make use of Theorem \ref{th:ck2}. This yields immediately that 

\begin{align}
\notag n^{-\frac{\alpha^{2}+\beta^{2}}{2}}I_3&=\frac{G(1+\frac{\alpha}{2}-i\frac{\beta}{2})^{2}G(1+\frac{\alpha}{2}+i\frac{\beta}{2})^{2}}{G(1+\alpha)^{2}}\\
&\qquad\times \int_{\frac{\log n}{n}\leq |\theta-\theta'|<  2t_0}g(\theta)g(\theta')|e^{i\theta}-e^{i\theta'}|^{-\frac{\alpha^{2}+\beta^{2}}{2}}e^{\mathit{o}(1)}d\theta d\theta'.
\end{align}
 
As the singularity $|e^{i\theta}-e^{i\theta'}|^{-\frac{\alpha^{2}+\beta^{2}}{2}}=(2\sin\frac{|\theta-\theta'|}{2})^{-\frac{\alpha^{2}+\beta^{2}}{2}}$ is integrable as $\theta\to \theta'$, and the error is uniform, we find

\begin{align}
\notag n^{-\frac{\alpha^{2}+\beta^{2}}{2}}I_3&\to\frac{G(1+\frac{\alpha}{2}-i\frac{\beta}{2})^{2}G(1+\frac{\alpha}{2}+i\frac{\beta}{2})^{2}}{G(1+\alpha)^{2}}\\
&\qquad \times \int_{0\leq |\theta-\theta'|<  2t_0}g(\theta)g(\theta')|e^{i\theta}-e^{i\theta'}|^{-\frac{\alpha^{2}+\beta^{2}}{2}}d\theta d\theta'.
\end{align}
 
Putting things together yields the claim. 
 
\end{proof}

\subsection{Asymptotics of the normalization constants} To prove Proposition \ref{prop:var}, we only need to calculate the asymptotics of the normalizing constants, i.e. $\E(f_{n,\alpha,\beta}^{(k)}(0))$ and $\E(f_{n,\alpha,\beta}(0))$.

\begin{lemma}\label{le:const}
For any fixed $k$, 
\begin{equation}
\lim_{n\to\infty}\E(f_{n,\alpha,\beta}^{(k)}(0))=e^{\frac{\alpha^{2}+\beta^{2}}{4}\sum_{j=1}^{k}\frac{1}{j}}
\end{equation}

\noindent and

\begin{equation}
\lim_{n\to\infty}n^{-\frac{\alpha^{2}+\beta^{2}}{4}}\E(f_{n,\alpha,\beta}(0))=\frac{G(1+\frac{\alpha}{2}-i\frac{\beta}{2})G(1+\frac{\alpha}{2}+i\frac{\beta}{2})}{G(1+\alpha)}
\end{equation}

\end{lemma}

\begin{proof}
By Heine-Szeg\"o (Theorem \ref{th:hs}), 

\begin{equation}
\E(f_{n,\alpha,\beta}^{(k)}(0))=D_{n-1}\left(e^{-\frac{1}{2}\sum_{j=1}^{k}\frac{1}{j}\left((\alpha-\beta i)e^{ij\phi}+(\alpha+i\beta)e^{-ij\phi}\right)}\right)
\end{equation}

\noindent and by the Strong Szeg\"o theorem (Theorem \ref{th:sz})

\begin{equation}
D_{n-1}\left(e^{-\frac{1}{2}\sum_{j=1}^{k}\frac{1}{j}\left((\alpha-\beta i)e^{ij\phi}+(\alpha+i\beta)e^{-ij\phi}\right)}\right)=e^{\frac{\alpha^{2}+\beta^{2}}{4}\sum_{j=1}^{k}\frac{1}{j}+\mathit{o}(1)}.
\end{equation}

For the second normalizing constant, one could note that it is a Selberg-Morris integral and can be written explicitly as a product of ratios of $\Gamma$-functions, but to avoid computations, we make use of Theorem \ref{th:fh}. We have $\E(f_{n,\alpha,\beta}(0))=D_{n-1}(|1-e^{i\phi}|^{\alpha}e^{\beta \mathrm{Im}\log(1-e^{i\phi})})$ which in the framework of Theorem \ref{th:fh} corresponds to $m=0$, $V=0$, $\alpha_0=\frac{\alpha}{2}$, and $\beta_0=-i\frac{\beta}{2}$ so that the theorem implies that

\begin{equation}
\lim_{n\to\infty}n^{-\frac{\alpha^{2}+\beta^{2}}{4}}\E(f_{n,\alpha,\beta}(0))=\frac{G(1+\frac{\alpha}{2}-i\frac{\beta}{2})G(1+\frac{\alpha}{2}+i\frac{\beta}{2})}{G(1+\alpha)}.
\end{equation}

\end{proof}

\subsection{Proof of Proposition \ref{prop:var}}

Putting together Corollaries \ref{cor:sigma1}, \ref{cor:sigma2}, and \ref{cor:sigma3} as well as Lemma \ref{le:const} with Lemma \ref{le:var}, we find

\begin{align}
\lim_{n\to\infty}\notag &\E\left(\left(\int_0^{2\pi}g(\theta)\mu_{n,\alpha,\beta}^{(k)}(d\theta)-\int_0^{2\pi}g(\theta)\mu_{n,\alpha,\beta}(d\theta)\right)^{2}\right)\\
\notag &=\frac{e^{\frac{1}{2}(\alpha^{2}+\beta^{2})\sum_{j=1}^{k}\frac{1}{j}}\int_{0}^{2\pi}\int_0^{2\pi}g(\theta)g(\theta')e^{\frac{(\alpha^{2}+\beta^{2})}{2}\sum_{j=1}^{k}\frac{1}{j}\cos (j(\theta-\theta'))}d\theta d\theta'}{\left(e^{\frac{\alpha^{2}+\beta^{2}}{4}\sum_{j=1}^{k}\frac{1}{j}}\right)^{2}}\\
\notag &\qquad -2\lim_{n\to\infty}\frac{n^{\frac{\alpha^{2}+\beta^{2}}{4}}\frac{G\left(1+\frac{\alpha}{2}-i\frac{\beta}{2}\right)G\left(1+\frac{\alpha}{2}+i\frac{\beta}{2}\right)}{G(1+\alpha)}e^{\frac{1}{4}(\alpha^{2}+\beta^{2})\sum_{j=1}^{k}\frac{1}{j}}}{ n^{\frac{\alpha^{2}+\beta^{2}}{4}}\frac{G(1+\frac{\alpha}{2}-i\frac{\beta}{2})G(1+\frac{\alpha}{2}+i\frac{\beta}{2})}{G(1+\alpha)}e^{\frac{\alpha^{2}+\beta^{2}}{4}\sum_{j=1}^{k}\frac{1}{j}}}\\
\notag &\qquad \times\int_0^{2\pi}\int_0^{2\pi}g(\theta)g(\theta')e^{\frac{1}{2}(\alpha^{2}+\beta^{2})\sum_{j=1}^{k}\frac{1}{j}\cos j(\theta-\theta')}d\theta d\theta'\\
&\qquad +\lim_{n\to\infty}\frac{n^{\frac{\alpha^{2}+\beta^{2}}{2}}\frac{G(1+\frac{\alpha}{2}-i\frac{\beta}{2})^{2}G(1+\frac{\alpha}{2}+i\frac{\beta}{2})^{2}}{G(1+\alpha)^{2}}}{\left(n^{\frac{\alpha^{2}+\beta^{2}}{4}}\frac{G(1+\frac{\alpha}{2}-i\frac{\beta}{2})G(1+\frac{\alpha}{2}+i\frac{\beta}{2})}{G(1+\alpha)}\right)^{2}}\\
\notag &\qquad \times \int_0^{2\pi}\int_0^{2\pi}g(\theta)g(\theta')|e^{i\theta}-e^{i\theta'}|^{-\frac{\alpha^{2}+\beta^{2}}{2}}d\theta d\theta'\\
\notag&=\int_{[0,2\pi]^2}g(\theta)g(\theta')\left(|e^{i\theta}-e^{i\theta'}|^{-\frac{\alpha^{2}+\beta^{2}}{2}}-e^{\frac{\alpha^{2}+\beta^{2}}{2}\sum_{j=1}^{k}\frac{1}{j}\cos j(\theta-\theta')}\right)d\theta d\theta'.
\end{align}

As this quantity is non-negative for all $k$ (it is a limit of variances), it tends to zero as $k\to\infty$ due to Fatou's lemma once we write the integral as a difference of two integrals.

\section{Proof of the main result}

We are now in a position to prove our main theorem. In the previous section, we proved that for a non-negative continuous function $g$

\begin{equation}
\E\left(\left(\int_0^{2\pi}g(\theta)\mu_{n,\alpha,\beta}^{(k)}(d\theta)-\int_0^{2\pi}g(\theta)\mu_{n,\alpha,\beta}(d\theta)\right)^{2}\right)\to 0
\end{equation}

\noindent as we first let $n\to\infty$ and then $k\to\infty$, so in particular, 

\begin{equation}
\int_0^{2\pi}g(\theta)\mu_{n,\alpha,\beta}^{(k)}(d\theta)-\int_0^{2\pi}g(\theta)\mu_{n,\alpha,\beta}(d\theta)\stackrel{d}{\to}0
\end{equation}

\noindent in the same limit. Thus if we are able to prove that 

\begin{equation}
\int_0^{2\pi}g(\theta)\mu_{n,\alpha,\beta}^{(k)}(d\theta)\stackrel{d}{\to }\int_0^{2\pi}g(\theta)\mu_{\sqrt{\alpha^{2}+\beta^{2}}}(d\theta)
\end{equation}

\noindent in the same limit, we will be done (for a detailed formulation of this type of argument see e.g. Theorem 4.28 in \cite{kallenberg}). To do this, we first prove the following lemma, which is just a corollary of the results of Diaconis and Shahshahani - (ie. Theorem \ref{th:ds} in this paper).

\begin{lemma}

For any fixed $k$, any $\alpha,\beta\in \R$, and any continuous function $g$ defined on the unit circle

\begin{equation}
\int_0^{2\pi}g(\theta)\mu_{n,\alpha,\beta}^{(k)}(d\theta)\stackrel{d}{\to }\int_0^{2\pi}g(\theta)\mu_{\sqrt{\alpha^{2}+\beta^{2}}}^{(k)}(d\theta)
\end{equation}

\noindent as $n\to\infty$ (for the definition of $\mu_{\sqrt{\alpha^{2}+\beta^{2}}}^{(k)}(d\theta)$ see the appendix).

\end{lemma}

\begin{proof}

Consider the function $F:\C^k\to \C$,

\begin{equation}
F(z_1,...,z_k)=\int_0^{2\pi}g(\theta)e^{-\frac{1}{2}\sum_{j=1}^k\frac{1}{\sqrt{j}}\left((\alpha-i\beta)z_j e^{-ij\theta}+(\alpha+i\beta)\overline{z_j}e^{ij\theta}\right)}d\theta.
\end{equation}

This is continuous as $g$ is bounded, so we see (by \cite[Theorem 4.27]{kallenberg}) that Theorem \ref{th:ds} implies that

\begin{align}
\notag F\left(\mathrm{Tr}U_n,...,\frac{1}{\sqrt{k}}\mathrm{Tr}U_n^k\right)=&\int_0^{2\pi}f_{n,\alpha,\beta}^{(k)}(\theta)g(\theta)\frac{d\theta}{2\pi}\\
&\stackrel{d}{\to}\int_0^{2\pi} e^{-\frac{1}{2}\sum_{j=1}^k\frac{1}{\sqrt{j}}\left((\alpha-i\beta)Z_j e^{-ij\theta}+(\alpha+i\beta)Z_j^*e^{ij\theta}\right)}d\theta\\
\notag &\stackrel{d}{=}\int_0^{2\pi} e^{\frac{\sqrt{\alpha^2+\beta^2}}{2}\sum_{j=1}^k\frac{1}{\sqrt{j}}\left(Z_j e^{ij\theta}+Z_j^*e^{-ij\theta}\right)}d\theta
\end{align}

\noindent as $n\to \infty$. Here $(Z_j)_j$ are i.i.d. standard complex Gaussians and we used again the fact that for any $\phi\in \R$, $(e^{i\phi}Z_j)_j\stackrel{d}{=}(-Z_j)_j$ as well as the fact that $(Z_j)_j\stackrel{d}{=}(Z_j^*)_j$. Now combining this with Lemma \ref{le:const} gives the desired result.
\end{proof}

As $\mu_{\sqrt{\alpha^{2}+\beta^{2}}}$ is defined to be the limit of $\mu_{\sqrt{\alpha^{2}+\beta^{2}}}^{(k)}$, this immediately implies that for continuous functions $g$, as we first let $n\to\infty$ and then $k\to\infty$,

\begin{equation}
\int_0^{2\pi}g(\theta)\mu_{n,\alpha,\beta}^{(k)}(d\theta)\stackrel{d}{\to}\int_0^{2\pi}g(\theta)\mu_{\sqrt{\alpha^{2}+\beta^{2}}}(d\theta)
\end{equation}

Putting things together, we conclude that

\begin{equation}
\int_0^{2\pi}g(\theta)\mu_{n,\alpha,\beta}(d\theta)\stackrel{d}{\to}\int_0^{2\pi}g(\theta)\mu_{\sqrt{\alpha^{2}+\beta^{2}}}(d\theta)
\end{equation}

\noindent which was what we wanted to prove.

\section{Discussion and open problems}

The main goal of this article was to prove a new type of geometric limit theorem describing the asymptotic behavior of the characteristic polynomial of a large random unitary matrix and thus linking random matrix theory to the theory of Gaussian multiplicative chaos. As noted in the introduction, to the author's knowledge, this is the first rigorous proof of such a link. From the point of view of random matrix theory, this connection sheds light on the global multifractal structure of the eigenvalues of a CUE matrix, and gives one new tools for studying some asymptotic properties of the eigenvalues. From the point of view of Gaussian multiplicative chaos, this is - to the author's knowledge - the first non-trivial model where Gaussian multiplicative chaos appears naturally. By non-trivial we mean here an approximation of a Gaussian field that is neither Gaussian nor a martingale, and appears naturally from other considerations. From either point of view, this connection suggests exciting new questions to explore and we discuss some of them here. 

\subsection{Other values of $\alpha$ and $\beta$}

Non-trivial multiplicative chaos measures $e^{\gamma X(\theta)-\frac{\gamma^{2}}{2}\E(X(\theta)^{2})}d\theta$ can be constructed for all values of $\gamma$. Our restriction to the $L^{2}$-phase i.e. $\alpha^{2}+\beta^{2}<2$ was due to the fact that we estimated variances. For $\alpha^{2}+\beta^{2}\geq 2$, these variances will blow up and the estimates would no longer be good. Moreover, the condition that $\alpha>-\frac{1}{2}$ was due to asymptotic analysis of the Toeplitz determinant being valid in this regime.

\vspace{0.3cm}

A natural question to ask is then can one go beyond these values of $\alpha$ and $\beta$. In the $L^{1}$-phase, namely where the martingale defining the multiplicative chaos measure is uniformly integrable (in our setting this means $\alpha^2+\beta^2<4$), one could expect that perhaps instead of estimating variances one could estimate moments of order $p$ with $1<p<2$. While this would seem to make the Toeplitz determinant approach impossible, perhaps there is a way to rely on variance estimates as is common in multiplicative chaos theory (there moments of order $p$ are often estimated using variance estimates in a clever way).

\vspace{0.3cm}

Going out of the $L^{1}$-phase, the construction of multiplicative chaos measures becomes much more challenging (it is no longer enough to normalize by the mean - see \cite{drsv1,drsv2,rvm} - and presumably one will need a different kind of approach in this regime. A related question is studying the maximum of $\log |p_n(\theta)|$. The conjecture of Fyodorov and Keating is that this should behave like the maximum of a log-correlated field (see \cite{mad,bdz,drz}). In the case of a log-correlated field, the multiplicative chaos measures play a role in understanding the behavior of the maximum. Again, analyzing this in the case of $\log |p_n(\theta)|$ will presumably require some other kind of approach.

\vspace{0.3cm}

It might also be possible to relax the $\alpha>-\frac{1}{2}$ condition to some degree. For example, in the case of a single Fisher-Hartwig singularity, the condition that $\mathrm{Re}(\alpha_0)>-\frac{1}{2}$ in Theorem \ref{th:fh} can be significantly relaxed - see \cite{ehrhardt}.

\vspace{0.3cm}

Another natural extension is to the case of complex $\alpha$ and $\beta$ (for simplicity, let us discuss the $L^{2}$ phase). Indeed as remarked in the appendix (see Remark \ref{rem:cplx}) complex Gaussian multiplicative chaos can be considered. Also asymptotics of Toeplitz determinants with complex parameters are known. The issue here is that for complex parameters, $\log D_{n-1}(\sigma_{3,\theta,\theta'})$ may have singularities for some values of $\theta,\theta'$ - see Theorem 1.8 in \cite{claeyskrasovsky}. That being said, these singularities should correspond to zeros in the asymptotics of $D_{n-1}(\sigma_{3,\theta,\theta'})$ (see Remark 1.9 in \cite{claeyskrasovsky}) so they should not be problematic.

\subsection{Other random matrix models}

Another natural question is what depends on the special structure of the CUE here. The author's guess is that perhaps this connection between multiplicative chaos and random matrix theory is quite universal. There are many random matrix models where the fluctuations of the characteristic polynomial are log-correlated Gaussian fields: the GUE, one-dimensional $\beta$-ensembles, the Ginibre ensemble, and random normal matrices\cite{fs,johansson,rv,ahm}. Moreover, for the GUE, there are results in \cite{krasovsky} corresponding to Theorem \ref{th:fh} here and one essentially needs to modify the results in \cite{claeyskrasovsky} to the GUE setting to prove a result as ours in the GUE case. Again in the GUE case presumably the $L^{2}$-phase is the simplest one and extending beyond that may be difficult. For conjectures regarding for example the maximum of the characteristic polynomial, see \cite{fs2}.

\vspace{0.3cm}

What is common for all of these mentioned models is that they are $\beta$-ensembles. Indeed, for when the Dyson index $\beta$ equals 2 in a one-dimensional model (on the real axis or the unit circle), our approach will lead to a Toeplitz or Hankel determinant whose analysis is presumably possible under suitable regularity conditions. In fact, generalizing our result to the case with a non-trivial potential on the unit circle (say analytic in a neighborhood of the unit circle) should not require much. The much more complicated question is what can one do in the two-dimensional case or when $\beta\neq 2$ and a Riemann-Hilbert approach might not exist.

\subsection{Limiting distribution of the total mass} 
We also point out a conjecture of Fyodorov and Bouchaud \cite{fb} on the total mass of the measure $\mu_\beta$. Combining this with our results suggests a conjecture on the asymptotic distribution of powers of the absolute value of the characteristic polynomial. There they provide an analytic continuation of the positive integer moments of the total mass and conjecture that the law of the total mass can be given in terms of negative powers of an exponentially distributed random variable. Such an analytic continuation is not unique (only finitely many positive integer moments exist so they can't determine the distribution) so this result is still an open question.

\section*{Appendix: Gaussian Multiplicative Chaos and Sobolev Spaces}

As mentioned in the introduction, Gaussian Multiplicative Chaos is a theory due to Kahane \cite{kahane}. One of the consequences of the theory is that it provides a method for exponentiating Gaussian fields with a logarithmic singularity in their covariance. More precisely, assume that one has a centered Gaussian field $(X(x))_{x\in A}$, where $A$ is say some open subset of $\R^{d}$ and the covariance kernel $C(x,y)=\E(X(x)X(y))$ has a logarithmic singularity: $C(x,y)\sim -\log |x-y|$ as $x\to y$. The goal is to construct a random measure of the form $e^{X(x)-\frac{1}{2}\E(X(x)^{2})}dx$. 

\vspace{0.3cm}

Due to the logarithmic singularity in the covariance of $X$, the field can not be realized as a random function, though it can be understood as a random distribution. In any event, the exponentiation can not be performed directly. The most natural way to do it is to regularize $X$ into a function say $X_n$ (where $X_n\to X$ in some suitable sense as $n\to\infty$), construct the measure $e^{X_n(x)-\frac{1}{2}\E(X_n(x)^{2})}dx$, and if this converges to some limiting measure, interpret the limit as $e^{X(x)-\frac{1}{2}\E(X(x)^{2})}dx$. 

\vspace{0.3cm}

One then is posed with the question of how should the field be regularized. One would naturally want the regularization to behave nicely with respect to a limiting procedure. One of the simplest random objects with rich limit theory is a martingale. This was Kahane's approach and his fundamental theorem is the following (see \cite{kahane,gmcrev}).

\begin{theorem}[Kahane]
Assume that for $x,y\in A$, $T>0$ and a continuous and bounded function $g$, 

\begin{equation}
C(x,y)=\log \frac{T}{|x-y|}+g(x,y),
\end{equation}

\noindent and assume that we have a decomposition

\begin{equation}
C(x,y)=\sum_{k=1}^{\infty}K_k(x,y),
\end{equation}

\noindent where $K_k$ are continuous and positive definite covariance kernels. Then if one defines on the same probability space the centered Gaussian random fields $(Y_k)_{k=1}^{\infty}$ , where $Y_k$ is independent of $Y_{k'}$ for $k\neq k'$ and $Y_k$ has covariance $K_k$, as well as the fields $X_n=\sum_{k=1}^{n}Y_k$ then for $\beta\in \R$, the measures 

\begin{equation}
M_{\beta,n}(dx)=e^{\beta X_n(x)-\frac{\beta^{2}}{2}\sum_{k=1}^{n}K_k(x,x)}dx
\end{equation}

\noindent converge almost surely in the space of Radon measures (with respect to the topology of weak convergence) to some random measure $M_\beta(dx)$. This measure is non-trivial for $\beta^{2}<2d$ and the zero measure for $\beta^{2}\geq 2d$. If all of the $K_k$ in the decomposition of $C$ are non-negative, the law of $M_\beta$ is independent of the specific decomposition.

\end{theorem}

Our interest will be in the field $X$ which can be viewed as the restriction of the whole plane Gaussian Free Field restricted to the unit circle, namely it has covariance $\E(X(\theta)X(\theta'))=-\frac{1}{2}\log |e^{i\theta}-e^{i\theta'}|$ (we choose the normalizing constant $\frac{1}{2}$ simply to be consistent in notation). To make precise sense of this object, we interpret it as an element of a Sobolev space.

\begin{definition}
For $s\in \R$, consider the space of formal Fourier series

\begin{equation}
\mathcal{H}^{s}=\left.\left\lbrace f\sim \sum_{k\in \Z}f_k e^{ik\theta}\right|\sum_{k\in \Z}(1+k^{2})^{s}|f_k|^{2}<\infty\right\rbrace
\end{equation}

\noindent with inner product

\begin{equation}
\langle f,g\rangle_s=\sum_{k\in \Z}(1+k^{2})^{s}f_k g_k^{*}.
\end{equation}

The subspace $\lbrace f\in \mathcal{H}^{s}|f_0=0\rbrace$ is denoted by $\mathcal{H}_0^{s}$.
\end{definition}

\begin{remark}
These are Hilbert spaces for all values of $s\in \R$. Moreover, for $s\geq 0$, they can be interpreted as subspaces of square integrable functions on the unit circle while for $s<0$ they are dual spaces of these and can be interpreted as spaces of generalized functions.
\end{remark}

One can then check that if $(Z_k)_{k=1}^{\infty}$ are i.i.d. standard complex Gaussians, then 

\begin{equation}
X:=\frac{1}{2}\sum_{k=1}^{\infty}\frac{1}{\sqrt{k}}\left(Z_k e^{ik\theta}+Z_k^{*}e^{-ik\theta}\right)
\end{equation}

\noindent is almost surely an element of $\mathcal{H}_0^{-s}$ for any $s>0$ and it has covariance kernel $-\frac{1}{2}\log |e^{i\theta}-e^{i\theta'}|$. Moreover, being a sum of i.i.d. Gaussian terms, this fits immediately into Kahane's theorem. Let us make the following definition:

\begin{definition}
Let $(Z_i)_{i=1}^{\infty}$ be i.i.d. standard complex Gaussians and 

\begin{equation}
X_n(\theta)=\frac{1}{2}\sum_{k=1}^{n}\frac{1}{\sqrt{k}}(Z_ke^{ik\theta}+Z_k^{*}e^{-ik\theta}).
\end{equation}

Moreover, let 

\begin{equation}
\mu_\beta^{(k)}(d\theta)=e^{\beta X_k(\theta)-\frac{\beta^{2}}{2}\E(X_k(\theta)^{2})}d\theta
\end{equation}

\noindent and 

\begin{equation}
\mu_\beta(d\theta)=\lim_{k\to\infty}\mu_\beta^{(k)}(d\theta)
\end{equation}

\noindent which exists for all $\beta\in \R$ (when the limit is in the topology of weak convergence) and is non-trivial for $|\beta|<2$.

\end{definition}

\begin{remark}
Note that the measures appearing in our case are $\mu_\beta$ for $|\beta|<\sqrt{2}$. This corresponds to the situation where $\E(\mu_\beta([0,2\pi))^{2})<\infty$ or "the $L^{2}$-phase".
\end{remark}

\begin{remark}
Note that we don't have the positivity of the covariances required for the uniqueness in Kahane's theorem, so it is not immediately clear that this measure is the same one gets through other constructions such as the one in \cite{ajks}. There have recently been generalizations to the construction of Kahane, see e.g. \cite{gmcrevisited,ongmc}. In particular, uniqueness questions relevant to our situation have been addressed in \cite{ongmc}. 
\end{remark}

\begin{remark}\label{rem:cplx}
We point out that it is natural to consider such objects also for a complex parameter $\beta$. In this case, these objects might not be complex measures: the total variation of the measure $e^{\beta X_n(\theta)-\frac{\beta^{2}}{2}\E(X_n(\theta)^{2})}d\theta$ is $e^{\mathrm{Re}(\beta) X_n(\theta)-\frac{\mathrm{Re}(\beta)^{2}-\mathrm{Im}(\beta)^{2}}{2}\E(X_n(\theta)^{2})}d\theta$. As $e^{\mathrm{Re}(\beta) X_n(\theta)-\frac{\mathrm{Re}(\beta)^{2}}{2}\E(X_n(\theta)^{2})}d\theta$ will converge to a non-trivial chaos measure (for small enough $\mathrm{Re}(\beta)$) it is reasonable to expect that for any $\beta$ with $\mathrm{Im}(\beta)\neq 0$, the $e^{\frac{\mathrm{Im}(\beta)^2}{2}\E(X_n(\theta)^2)}$-term will cause the total variation of the limit $e^{\beta X(\theta)-\frac{\beta^2}{2}\E(X(\theta)^2)}d\theta/2\pi$ to be almost surely infinite, so perhaps it can't be understood as a complex measure.  One possibility for a natural interpretation of $e^{\beta X(\theta)-\frac{\beta^{2}}{2}\E(X(\theta)^{2})}$ is as a random distribution, for example an element of $\mathcal{H}^{-s}$ for large enough $s>0$. Much of the reasoning goes through here too - one can use martingale arguments etc. For further results on complex Gaussian multiplicative chaos, see for example \cite{bmj,gmccomplex}. 
\end{remark}

\section*{Acknowledgements}
The author wishes to thank Antti Kupiainen, Eero Saksman, Yan Fyodorov, and Nicholas Simm for useful discussions. This work was partly supported by the Academy of Finland. The author also wishes to thank two anonymous referees for their careful reading of the manuscript and helpful remarks.


\begin{thebibliography}{99}

\bibitem{ahm} Y. Ameur, H. Hedenmalm, and N. Makarov: Fluctuations of eigenvalues of random normal matrices. Duke Math. J. 159 (2011), no. 1, 31–81. 
\bibitem{ajks} K. Astala, P. Jones, A. Kupiainen, E. Saksman. Random conformal weldings. Acta Math. 207 (2011), no. 2, 203–254.
\bibitem{bm} E. Bacry, A. Kozhemyak, J.-F. Muzy: Continuous cascade models for asset returns.
J. Econom. Dynam. Control 32 (2008), no. 1, 156–199.
\bibitem{bmj} J. Barral, X. Jin, B. Mandelbrot. Uniform convergence for complex [0,1]-martingales. Ann. Appl. Probab. 20 (2010), no. 4, 1205–1218.
\bibitem{basor} E. Basor. Asymptotic formulas for Toeplitz determinants. Trans. Amer. Math. Soc. 239 (1978), 33–65
\bibitem{bss} N. Berestycki, S. Sheffield, and X. Sun: Liouville quantum gravity and the Gaussian free field. Preprint,  arXiv:1410.5407.
\bibitem{bhny} P. Bourgade, C. P. Hughes, A. Nikeghbali, M. Yor. The characteristic polynomial of a random unitary matrix: a probabilistic approach. Duke Math. J. 145 (2008), no. 1, 45–69.
\bibitem{bdz} M. Bramson, J. Ding, O. Zeitouni. Convergence in law of the maximum of the two-dimensional discrete Gaussian free field. Preprint arXiv:1301.6669
\bibitem{bd} D. Bump, P. Diaconis. Toeplitz minors.  J. Combin. Theory Ser. A 97 (2002), no. 2, 252–271. 
\bibitem{cld} D. Carpentier, P. Le Doussal. Glass transition of a particle in a random potential, front selection in non linear RG and entropic phenomena in Liouville and SinhGordon models. Phys. Rev. E 63, 026110 (2001) 
\bibitem{cnn} R. Chhaibi, J. Najnudel, A. Nikeghbali. A limiting random analytic function related to the CUE. Preprint arXiv:1403.7814.  
\bibitem{claeyskrasovsky} T. Claeys, I. Krasovsky. Toeplitz determinants with merging singularities. Preprint arXiv:1403.3639.
\bibitem{de} F. David and B. Eynard: Planar maps, circle patterns and 2D gravity. Ann. Inst. Henri Poincaré D 1 (2014), no. 2, 139–183. 
\bibitem{dkrv} F. David, A. Kupiainen, R. Rhodes, and V. Vargas: Liouville Quantum Gravity on the Riemann sphere. Preprint arXiv:1410.7318.
\bibitem{deift} P. Deift, A. Its, and I. Krasovsky. On the asymptotics of a Toeplitz determinant with singularities. Preprint arXiv:1206.1292 
\bibitem{deift2} P. Deift, A. Its, and I. Krasovsky: Asymptotics of Toeplitz, Hankel, and Toeplitz+Hankel determinants with Fisher-Hartwig singularities. Ann. of Math. (2) 174 (2011), no. 2, 1243–1299. 
\bibitem{ds} P. Diaconis and M. Shahshahani. On the eigenvalues of random matrices. Studies in applied probability. J. Appl. Probab. 31A (1994), 49–62.
\bibitem{drz} J. Ding, R. Roy, and O. Zeitouni: Convergence of the centered maximum of log-correlated Gaussian fields, preprint arXiv:1503.04588.
\bibitem{dupshef} B. Duplantier, S. Sheffield. Liouville quantum gravity and KPZ. Invent. Math. 185 (2011), no. 2, 333–393.
\bibitem{drsv1} B. Duplantier, R. Rhodes, S. Sheffield, V. Vargas. Critical Gaussian multiplicative chaos: Convergence of the derivative martingale. Ann. Probab. 42 (2014), no. 5, 1769–1808
\bibitem{drsv2} B. Duplantier, R. Rhodes, S. Sheffield, V. Vargas.  Renormalization of critical Gaussian multiplicative chaos and KPZ relation. Comm. Math. Phys. 330 (2014), no. 1, 283–330.
\bibitem{dyson}  F.M. Dyson. The threefold way. Algebraic structure of symmetry groups and ensembles in quantum mechanics. J. Math. Phys. 3: 1199  (1962)
\bibitem{ehrhardt1} T. Ehrhardt: A status report on the asymptotic behavior of Toeplitz determinants with Fisher-Hartwig singularities. Recent advances in operator theory (Groningen, 1998), 217–241, Oper. Theory Adv. Appl., 124, Birkhäuser, Basel, 2001.
\bibitem{ehrhardt} T. Ehrhardt, B. Silbermann. Toeplitz determinants with one Fisher-Hartwig singularity. J. Funct. Anal. 148 (1997), no. 1, 229–256.
\bibitem{fh} M. E. Fisher, R. E. Hartwig. Toeplitz determinants. Some applications, theorems and conjectures. Adv. Chem. Phys., vol. 15 (1968).
\bibitem{fb} Y. V. Fyodorov, J.-P. Bouchaud. Freezing and extreme-value statistics in a random energy model with logarithmically correlated potential. J. Phys. A 41 (2008), no. 37, 372001, 12 pp.
\bibitem{fk} Y. V. Fyodorov, J. P. Keating. Freezing transitions and extreme values: random matrix theory, and disordered landscapes. Philos. Trans. R. Soc. Lond. Ser. A Math. Phys. Eng. Sci. 372 (2014), no. 2007, 20120503, 32 pp. 
\bibitem{fs} Y. V. Fyodorov, B.A. Khoruzhenko, N.J. Simm. Fractional Brownian motion with Hurst index H=0 and the Gaussian Unitary Ensemble. Preprint arXiv:1312.0212
\bibitem{fs2} Y.V. Fyodorov and N.J. Simm: On the distribution of maximum value of the characteristic polynomial of GUE random matrices.  Preprint arXiv:1503.07110.
\bibitem{hko} C.P. Hughes, J. P. Keating, N. O'Connell. On the characteristic polynomial of a random unitary matrix.  Comm. Math. Phys. 220 (2001), no. 2, 429–451.
\bibitem{johansson} K. Johansson, On fluctuations of eigenvalues of random Hermitian matrices. Duke Math. J. 91 (1998), no. 1, 151–204.
\bibitem{kahane}  J.-P. Kahane. Sur le chaos multiplicatif.  Ann. Sci. Math. Québec 9 (1985), no. 2, 105–150.
\bibitem{kallenberg} O. Kallenberg. Foundations of modern probability. Second edition. Probability and its Applications (New York). Springer-Verlag, New York, 2002. xx+638 pp. ISBN: 0-387-95313-2.
\bibitem{kallenberg2} O. Kallenberg: Random measures. Third edition. Akademie-Verlag, Berlin; Academic Press, Inc. [Harcourt Brace Jovanovich, Publishers], London, 1983. 187 pp. ISBN: 0-12-394960-2.
\bibitem{ks} J. P. Keating, N. C. Snaith. Random matrix theory and $\zeta(1/2+it)$. Comm. Math. Phys. 214 (2000), no. 1, 57–89.
\bibitem{kpz} V. G. Knizhnik, A. M. Polyakov, A. B.
Zamolodchikov. Fractal structure of 2d-quantum gravity. Mod. Phys. Lett. A 3, 819 (1988).
\bibitem{kol} A. Kolmogorov: A refinement of previous hypotheses concerning
the local structure of turbulence in a viscous incompressible fluid
at high Reynolds number. J. Fluid Mech. 13, 82–85 (1962).
\bibitem{krasovsky} I. Krasovsky: Correlations of the characteristic polynomials in the Gaussian unitary ensemble or a singular Hankel determinant. Duke Math. J. 139 (2007), no. 3, 581–619. 
\bibitem{gmccomplex} H. Lacoin, R. Rhodes, V. Vargas. Complex Gaussian multiplicative chaos.  Comm. Math. Phys. 337 (2015), no. 2, 569–632.
\bibitem{lenard} A. Lenard. Some remarks on large Toeplitz determinants.  Pacific J. Math. 42 (1972), 137–145. 
\bibitem{mad} T. Madaule. Maximum of a log-correlated Gaussian field. Preprint arXiv:1307.1365.
\bibitem{rvm} T. Madaule, R. Rhodes, V. Vargas. Glassy phase and freezing of log-correlated Gaussian potentials. Preprint arXiv:1310.5574.
\bibitem{qle} J. Miller and S. Sheffield: Quantum Loewner Evolution. Preprint,  arXiv:1312.5745.
\bibitem{obu} A. Obukhov: Some specific features of atmospheric turbulence.
J. Fluid Mech. 13, 77–81 (1962).
\bibitem{kpzrv} R. Rhodes, V. Vargas. KPZ formula for log-infinitely divisible multifractal random measures, ESAIM Probability and Statistics, 15 (2011) 358-371.
\bibitem{gmcrev} R.Rhodes, V. Vargas. Gaussian multiplicative chaos and applications: a review.  Probab. Surv. 11 (2014), 315–392. 
\bibitem{rv} B. Rider, B. Vir\'ag. The noise in the circular law and the Gaussian free field. Int. Math. Res. Not. IMRN 2007, no. 2, Art. ID rnm006, 33 pp.
\bibitem{gmcrevisited} R. Robert, V. Vargas. Gaussian multiplicative chaos revisited. Ann. Probab. 38 (2010), no. 2, 605–631.
\bibitem{ongmc} A. Shamov. On Gaussian multiplicative chaos. Preprint arXiv:1407.4418.
\bibitem{shefweld} S. Sheffield. Conformal weldings of random surfaces: SLE and the quantum gravity zipper. Preprint arXiv:1012.4797.
\bibitem{simon} B. Simon: The sharp form of the strong Szeg\"o theorem. Geometry, spectral theory, groups, and dynamics, 253–275, Contemp. Math., 387, Amer. Math. Soc., Providence, RI, 2005.
\bibitem{sze} G. Szeg\"o: On certain Hermitian forms associated with the Fourier series of
a positive function, Comm. S\'em. Math. Univ. Lund 1952 (1952), Tome Supplementaire,
228–238.
\bibitem{widom}  H. Widom. Toeplitz determinants with singular generating functions. Amer. J. Math. 95 (1973), 333–383.


\end{thebibliography}
\end{document}